\documentclass[12pt]{amsart}

\usepackage{graphicx}
\usepackage{subcaption}
\usepackage{xfrac}
\usepackage{faktor}
\usepackage{tikz-cd}
\usepackage{mathrsfs}
\usepackage{hyperref}
\usepackage{amsmath}
\usepackage{amssymb}
\usepackage{url}
\usepackage{comment}
\usepackage{xcolor}
\usepackage[T1]{fontenc}
\usepackage[utf8]{inputenc}

\hypersetup{
	bookmarks=true,
	unicode=true,
	colorlinks=true,
	citecolor=black,
	linkcolor=black,
	urlcolor=black,
	plainpages=false,
	pdfpagelabels=true
}

\newtheorem{teo}{Theorem}[section]
\newtheorem{theorem}[teo]{Theorem}

\newtheorem{lemma}[teo]{Lemma}
\newtheorem{prop}[teo]{Proposition}
\newtheorem{corollary}[teo]{Corollary}
\theoremstyle{definition}
\newtheorem{definition}[teo]{Definition}
\newtheorem{example}[teo]{Example}
\theoremstyle{remark}
\newtheorem{remark}[teo]{Remark}
\numberwithin{figure}{section}

\newcommand{\CC}{\mathcal{C}}

\newcommand{\Dc}{\mathcal{D}}

\newcommand{\id}{\mathrm{id}}

\newcommand{\Hom}{\mathrm{Hom}}
\newcommand{\Word}{\mathrm{Word}}
\newcommand{\Sub}{\mathrm{Sub}}
\newcommand{\FinInj}{\mathrm{FinInj}}

\newcommand{\Ob}{\mathrm{Ob}}

\newcommand{\Fork}{\mathcal{F}\!\mathit{ork}}

\begin{document}

\title[Subword representations]{Subword representations and weak hypercube dimension for acyclic categories}

\author[Isaac Carcacía-Campos]{Isaac Carcacía-Campos}
\address{Departamento de Matem\'aticas, Universidade de Santiago de Compostela, 15782-Spain}
\email{isaac.c.campos@usc.es}

\begin{abstract}
We introduce a categorical analogue of weak hypercube representations of finite posets by means of faithful embeddings into categories of subwords of finite words.
For finite acyclic categories, we characterize those admitting such a weak subword representation: they are precisely the monic categories whose hom-sets carry a left-compatible local total order.
The proof is constructive and gives an explicit word representation. We also introduce a query game for categories, generalizing a Boolean query game for posets, and show how winning sets produce explicit word representations and hence upper bounds for the weak word dimension.
\end{abstract}


\maketitle

\section{Introduction}

The abstract notion of a partial order encompasses many different kinds of relations. In mereology, for instance, it is usually assumed that the relation \textit{being a part of} is a partial order, possibly satisfying additional axioms \cite{Mereology}. There are, however, some caveats to this point of view.

For example, consider the word \textit{attack}. The one-letter words \textit{a} and \textit{t} are subwords of \textit{attack}, but each occurs in two distinct ways: the letter \textit{a} can be embedded in positions \(1\) or \(4\), and \textit{t} in positions \(2\) or \(3\). By contrast, \textit{c} and \textit{k} occur only once. Thus, a mere yes/no relation of parthood loses information about the number of possible embeddings. Similar phenomena arise when modelling processes with repeated homogeneous parts, such as music \cite{DiRienzo}.

This suggests replacing partially ordered sets by richer structures capable of recording not only whether one object is part of another, but also how many different ways this can happen. In this paper we use small acyclic categories for this purpose. 

Such a category may be viewed as a directed acyclic multigraph equipped with a transitive structure given by composition of morphisms in which we impose some relations \cite{Polygraph}. Every poset is a special case: it is a skeletal acyclic category with at most one morphism between any two objects, where reflexivity and transitivity are transmuted into identity morphisms and composition.

The central example for us is the category \(\Sub_w\) of subwords of a fixed word \(w\) over an alphabet \(\Sigma\). Its objects are the words that occur as subwords of \(w\), and its morphisms are the possible subword inclusions, i.e. injective order-preserving maps between positions that respect the labels.

A particularly important case occurs when all letters of \(w\) are distinct. Then \(\Sub_w\) is isomorphic to the Boolean lattice \(2^{[n]}\), where \(n\) is the length of \(w\). Thus categories of subwords generalize Boolean lattices by allowing repeated letters, and repeated letters produce parallel morphisms.

In order theory, several notions of dimension have been studied for finite posets. The classical one is the \emph{Dushnik--Miller dimension}, or order dimension, introduced in \cite{PartiallyOrderedSets}. It is the minimum number of linear extensions whose intersection is the given poset. Equivalently, it is the minimum number of chains whose product contains the poset as an order-embedded subposet \cite[chapter 10]{ORE}.

Another invariant is the \(2\)-dimension, or hypercube dimension \cite{optimalemb,Trotter_Embedding}. In its standard form, it is the least integer \(n\) for which the poset embeds as a subposet of the Boolean lattice \(2^{[n]}\). Equivalently, it is the least \(n\) for which there exists an order-preserving and order-reflecting map into \(2^{[n]}\).

In this paper we work with a weaker version. The \emph{weak \(2\)-dimension} of a finite poset \(P\), denoted \(\dim^{\mathrm{w}}_2(P)\), is the least integer \(n\) for which there exists an injective order-preserving map \[P \longrightarrow 2^{[n]}.\] Unlike the usual \(2\)-dimension, this weak version does not require incomparability to be reflected. For example, an antichain of size \(m\) has weak \(2\)-dimension \(\lceil \log_2 m\rceil\), while its usual \(2\)-dimension is governed by Sperner's theorem \cite{Sperner}.

The categorical notion studied here extends precisely this weak version. We define the \emph{weak word dimension} of a small acyclic category \(\CC\) to be the smallest non-negative integer \(n\) for which there exist an alphabet \(\Sigma\), a word \(w\) of length \(n\), and a faithful functor \[\CC \longrightarrow \Sub_w\] which is injective on objects. When \(\CC\) is a poset, this reduces to the weak \(2\)-dimension.

It is useful to keep in mind a stronger variant. One may ask the representing functor to be fully faithful; for posets, this recovers the usual \(2\)-dimension. The present paper focuses on the weak version. The strong version is substantially more rigid, because it requires that every subword inclusion between representing words come from a morphism of the original category.

Our main result gives a complete characterization of finite acyclic categories with finite weak word dimension. They are exactly the finite acyclic categories that are \emph{monic}, meaning that all morphisms are monomorphisms, and \emph{left locally totally ordered}, meaning that each hom-set carries a total
order compatible with postcomposition. The proof is constructive and gives an explicit algorithm for building a representing word.

We also study several basic constructions. In particular, we show that the class of finite acyclic categories satisfying the embedding criterion is stable under slices, certain functor categories, and finite limits.

Finally, we introduce a query game for categories related to the weak word dimension. A player secretly chooses a morphism \(f\colon x\to y\). The opponent selects a finite set \(B\) of test objects and first asks, for each \(b\in B\), the numbers of morphisms from \(b\) to \(x\) and to \(y\). If these data determine \(x\) and \(y\), the opponent then asks, for each \(b\in B\) and each \(h\colon b\to X\), for the composite \(f\circ h\).

We prove that a set \(B\) is winning exactly when it is both a \emph{cardinal separator}, distinguishing objects by the cardinalities of hom-sets from \(B\), and a \emph{separator}, distinguishing parallel morphisms by precomposition with morphism from objects of \(B\). In a poset the second condition is vacuous, and the first reduces to the Boolean query game: one identifies an element by asking which elements of a chosen subset lie below it.

A winning set supplies the data needed to construct a weak word representation and therefore gives an explicit upper bound for the weak word dimension. We illustrate the constructions with several examples.

\section{Dimensions of posets}\label{sec:poset-dim}

Let \(P\) be a finite poset. The dimension of a poset was introduced by Dushnik and Miller \cite{PartiallyOrderedSets}. We denote the order dimension of \(P\) by \(\dim(P)\). It is the smallest integer \(d\) for which there exist linear extensions \(L_1,\dots,L_d\) of \(P\) whose intersection is \(P\). Equivalently, \(P\) embeds, in the order-theoretic sense, into a product of \(d\) chains \cite[chapter 10]{ORE}. For a comprehensive account, see \cite{Trotter92} and to see relations with other combinatorial problems with posets \cite{Winter2024_1000173294}.

There is also a family of related invariants parametrized by positive integers. For \(k\geq 2\), the \emph{\(k\)-dimension} of \(P\), denoted \(\dim_k(P)\), is the least integer \(n\) such that \(P\) embeds as a subposet of the product of \(n\) chains of cardinality \(k\). For \(k=2\), this is the
problem of embedding \(P\) into a hypercube:
\[
2^{[n]} \cong \{0,1\}^n,
\] ordered coordinatewise, or equivalently by inclusion of subsets. This invariant is often called the \emph{\(2\)-dimension} or \emph{hypercube dimension}.

The order dimension and the \(k\)-dimension satisfy
\[
\dim(P)\leq \dim_k(P).
\]
Indeed, an embedding into a product of chains of cardinality \(k\) is in particular an embedding into a product of chains. The converse need not hold.

For our categorical purposes, we shall also use a weak version.

\begin{definition}
Let \(P\) be a finite poset. The \emph{weak \(2\)-dimension} of \(P\), denoted \(\dim^{\mathrm{w}}_2(P)\), is the least integer \(n\) such that there exists an injective order-preserving map
\[
P\longrightarrow 2^{[n]}.
\]
Thus \(\dim^{\mathrm{w}}_2(P)\) only requires preservation of order, whereas the usual \(2\)-dimension \(\dim_2(P)\) requires an order embedding, i.e. both preservation and reflection of order.
\end{definition}

The distinction is important. If \(P\) embeds into \(2^{[n]}\) in the usual sense, then it also admits an injective order-preserving map into \(2^{[n]}\).
Hence
\[
\dim^{\mathrm{w}}_2(P)\leq \dim_2(P).
\]
The inequality can be strict.

\begin{example}
Let \(A_m\) be an antichain of \(m\) elements, with \(m\geq 2\). For the Dushnik--Miller dimension we have
\[
\dim(A_m)=2.
\]
Indeed, choose any linear order \(L\) of the elements and let \(L'\) be its reverse. The intersection of \(L\) and \(L'\) contains only the reflexive comparabilities, hence it is the antichain.

For the weak \(2\)-dimension, no incomparabilities need to be reflected. Thus it is enough to assign distinct subsets of an \(n\)-element set to the \(m\) points. Therefore
\[
\dim^{\mathrm{w}}_2(A_m)=\lceil \log_2 m\rceil.
\]

By contrast, for the usual \(2\)-dimension, the image of an antichain must be an antichain in the Boolean lattice. By Sperner's theorem \cite{Sperner}, the largest antichain in \(2^{[n]}\) has cardinality
\(\binom{n}{\lfloor n/2\rfloor}\). Hence
\[
\dim_2(A_m)=
\min\left\{n:\binom{n}{\lfloor n/2\rfloor}\geq m\right\}.
\]
Thus the weak and strong versions differ already on antichains.
\end{example}

In the rest of the paper, our categorical constructions extend the weak \(2\)-dimension. The strong version, corresponding to fully faithful representations, is more rigid and will not be characterized here.

\section{Acyclic categories as combinatorial objects}

We briefly recall the basic notions of category theory and fix some notation. For a thorough introduction, we refer the reader to the standard literature \cite{Basic_Category,CatWorking}. Recall that a \emph{category} \(\mathcal{C}\) consists of a collection of objects and, for each ordered pair of objects \(c,d\), a collection of morphisms, or arrows, \(f\colon c\to d\), together with an associative composition operation
\[
(g,f)\longmapsto g\circ f\colon c\to e
\] for all morphisms \(f\colon c\to d\) and \(g\colon d\to e\), and for each object \(c\) an identity morphism \(\id_c\colon c\to c\) that acts as a unit for composition.

A category is \emph{small} if its collections of objects and morphisms are sets. In this paper we work exclusively with small categories.

A \emph{functor} \(F\colon\mathcal{C}\to\mathcal{D}\) between categories sends objects to objects and morphisms to morphisms, preserving identities and composition: 
\[F(\id_c)=\id_{F(c)},\qquad F(g\circ f)=F(g)\circ F(f).\]
We focus on acyclic categories as generalizations of posets, as usual in the literature \cite[chapter 10]{Kozlov}; see also \cite{Euler_cat,HomotopicDistance} for related examples of how to apply ideas from posets to small (acyclic) categories.

\begin{definition}[Acyclic category]
A small category \(\mathcal{C}\) is \emph{acyclic} if for every pair of distinct objects \(c\neq d\),
\[
\Hom(c,d)\neq\varnothing \quad\Longrightarrow\quad \Hom(d,c)=\varnothing,
\]
and for every object \(c\) the only endomorphism is the identity:
\[
\Hom(c,c)=\{\id_c\}.
\]
\end{definition}

Acyclicity forbids directed cycles of positive length: there are no
non-trivial loops \(c\to c\) and no non-trivial cyclic chains of morphisms
\[
c_1\to c_2\to\cdots\to c_n\to c_1.
\]
The second condition is not a consequence of excluding non-trivial endomorphisms alone: in a general category a cycle may compose to an identity, as happens for isomorphisms. Thus acyclicity rules out both non-trivial endomorphisms and cycles between distinct objects.

\begin{definition}[Weak and strong embeddings of categories]
A functor \(F\colon\mathcal{C}\to\mathcal{D}\) is a \emph{weak embedding} if it is faithful and injective on objects, that is:
\begin{itemize}
\item for every pair of objects \(c,d\) in \(\mathcal C\), the map
\[
F_{c,d}\colon \Hom_{\mathcal C}(c,d)\to
\Hom_{\mathcal D}(F(c),F(d))
\]
is injective;
\item for objects \(c\neq d\), one has \(F(c)\neq F(d)\).
\end{itemize}

A functor \(F\colon\mathcal{C}\to\mathcal{D}\) is a \emph{strong embedding} if it is fully faithful and injective on objects, hence the image of the functors is a full subcategory.
\end{definition}

\begin{remark}
For posets viewed as categories, a weak embedding is precisely an injective order-preserving map. A strong embedding is an order embedding: it preserves and reflects the order. Thus strong embeddings recover the usual notion of subposet embedding, whereas weak embeddings correspond to the weak representations studied in this paper.
\end{remark}

\begin{remark}
In many categorical settings one works with isomorphism classes of objects, because a category can contain many isomorphic copies of the same object without changing the essential structure. In an acyclic category, however, isomorphisms are necessarily identities. Indeed, if there exist morphisms \(f\colon c\to d\) and \(g\colon d\to c\), acyclicity forces \(c=d\), and then the only endomorphism is \(\id_c\). Consequently, every isomorphism is trivial, and two distinct objects cannot be isomorphic.
\end{remark}

\subsection{Linear extensions of acyclic categories}

A finite acyclic category admits an analogue of a topological ordering of a directed acyclic graph: a total order of the objects compatible with the existence of morphisms. More precisely, if there is a morphism \(c\to d\), then \(c\) must appear before \(d\) in the order.

For a small acyclic category \(\CC\), the relation ``there exists a morphism \(c\to d\)'' is a partial order on \(\Ob(\CC)\). Reflexivity follows from identity morphisms, transitivity from composition, and antisymmetry from acyclicity. We denote this poset by \(\Pi(\CC)\).

A \emph{linear extension} of \(\CC\) is a linear extension of the poset \(\Pi(\CC)\), i.e. a total order \(\preceq\) on \(\Ob(\CC)\) such that
\[
\Hom(c,d)\neq\varnothing \quad\Longrightarrow\quad c\preceq d.
\]

\begin{lemma}[{\cite[section 10.1.2]{Kozlov}}]\label{lem:linear-extension}
Every acyclic category admits a linear extension of its objects.
\end{lemma}

\section{The category of subwords}

Fix an alphabet \(\Sigma\). A \emph{word} of length \(n\) is a function \(w\colon [n]\to\Sigma\), where \([n]=\{1,\dots,n\}\) and \([0]=\varnothing\). We will usually write the word by appending all the letters in the usual order:
\[
w=w(1)\cdots w(n).
\]

\begin{definition}
Given two words \(w\colon[n]\to\Sigma\) and \(w'\colon[m]\to\Sigma\), a \emph{subword inclusion} is an injective order-preserving map \(f\colon[m]\to[n]\) such that
\[
w\circ f=w'.
\]
\end{definition}

\begin{remark}
We will represent a morphism \(f\colon[m]\to[n]\) as
\[
(f(1),f(2),\dots,f(m)).
\]
For example, the map \(f\colon[3]\to[10]\) given by \(f(1)=3\),
\(f(2)=6\), and \(f(3)=10\) will be represented by \((3,6,10)\).
\end{remark}

The class of all words over \(\Sigma\) becomes a small category \(\Word(\Sigma)\), whose morphisms are the subword inclusions. For a fixed word \(w\), we denote by \(\Sub_w\) the full subcategory of \(\Word(\Sigma)\) spanned by all words that admit a morphism into \(w\).

\begin{remark}
The definition of a subword as an order-preserving injective map, i.e. a subsequence, permits non-contiguous occurrences. For instance, in the word \(abc\), the letters \(a\) and \(c\) form the subword \(ac\), despite the gap. This may conflict with everyday linguistic intuition, where a ``part of'' a
word is often understood as a contiguous substring.
\end{remark}

The category \(\Sub_w\) is acyclic: any endomorphism is the identity, and any pair of morphisms going back and forth forces the two objects to be identical.

The length of a word provides a canonical linear extension into the natural numbers
\[
\mathbb{N}=\{0<1<2<\cdots\}.
\]
Indeed, if there is a subword inclusion \(w'\to w\), then the domain word has length at most the length of the codomain word, and equality holds only for identities. Thus the function
\[
\ell\colon\Ob(\Word(\Sigma))\to\mathbb N
\]
sending a word to its length is a linear extension. This property is inherited by any full subcategory \(\Sub_w\).

\begin{prop}
If \(w\colon[n]\to\Sigma\) is injective, then \(\Sub_w\) is isomorphic to the Boolean lattice \(2^{[n]}\) of subsets of an \(n\)-element set.
\end{prop}

\begin{proof}
Map a subword \(w'\colon[m]\to\Sigma\) to its image \(w'([m])\subseteq w([n])\). Because all letters are distinct, this correspondence is bijective and strictly preserves the subword relation.
\end{proof}

Thus the Boolean lattice representation of posets appears as the special case where the word has no repeated letters. When letters may repeat, the category \(\Sub_w\) becomes richer and can accommodate parallel morphisms.

\begin{example}\label{ex:aa}
Consider the word \(w=aa\). Its subwords are \(\varnothing\), \(a\), and \(aa\). There are two distinct inclusions of \(a\) into \(aa\), namely the maps sending the single position to position \(1\) or to position \(2\). The category \(\Sub_{aa}\) is depicted as follows:
\[
\begin{tikzcd}
aa & \\
a \arrow[u, "i_1", bend left] \arrow[u, "i_2"', bend right] & \\
\varnothing \arrow[u, "i"] &
\end{tikzcd}
\]
where \(i_1\circ i=i_2\circ i\).
\end{example}

\begin{example}
Let \(w=aba\). The subwords are
\[
\varnothing,\quad a,\quad b,\quad ab,\quad ba,\quad aa,\quad aba.
\]
The category \(\Sub_{aba}\) has the following diagram, with the commutations given by inclusions of subwords:
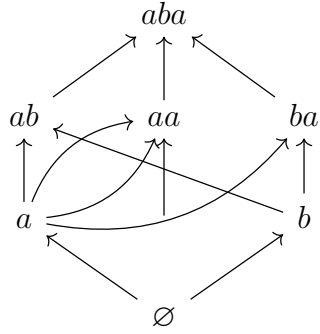
\[
\begin{tikzcd}
& aba & \\
ab \arrow[ru] & aa \arrow[u] & ba \arrow[lu] \\
a \arrow[u] \arrow[rru, bend right] \arrow[ru, bend left] \arrow[ru, bend right]
& {} \arrow[u]
& b \arrow[u] \arrow[llu] \\
& \varnothing \arrow[lu] \arrow[ru] &
\end{tikzcd}
\]
\end{example}

\subsection{Differences with the Boolean case}
\label{sec:differences-boolean}

As soon as letters may repeat, the structure changes considerably. We collect some of the most salient differences.

\subsubsection{Initial and terminal objects}

Every \(\Sub_w\) has an initial object, namely the empty word. In the Boolean case there is also a terminal object, the whole word \(w\), because there is exactly one inclusion of any subword into \(w\).

With repeated letters, \(w\) is no longer terminal. Indeed, if a letter occurs at least twice in \(w\), then the one-letter word corresponding to that letter has at least two distinct morphisms into \(w\). Since any terminal object would have to receive a morphism from \(w\), and \(w\) has maximal length, the only possible terminal object is \(w\) itself. Hence \(\Sub_w\) has a terminal object if and only if every letter occurring in \(w\) occurs exactly once.

\subsubsection{Products and coproducts}

A Boolean lattice has all finite products, given by intersection, and
coproducts, given by union. In contrast, the presence of parallel morphisms destroys binary products and coproducts in \(\Sub_w\).

We show both phenomena using Example~\ref{ex:aa}. First, the object \(aa\) has no product with itself in \(\Sub_{aa}\). Suppose that a product \(P=aa\times aa\) exists, with projections \(p_1,p_2\colon P\to aa\). Applying the universal property to \((\id_{aa},\id_{aa})\), there must be a unique morphism \(u\colon aa\to P\) such that
\[
p_1\circ u=\id_{aa},
\qquad
p_2\circ u=\id_{aa}.
\]
Since there is no morphism \(aa\to\varnothing\) or \(aa\to a\), the morphism \(u\) can only land in \(aa\). Thus \(P=aa\) and
\(p_1=p_2=\id_{aa}\).

Now take \(Q=a\) and the two distinct inclusions
\(i_1,i_2\colon a\to aa\). The universal property would require a unique morphism \(h\colon a\to aa\) such that
\[
p_1\circ h=i_1,
\qquad
p_2\circ h=i_2.
\]
Since \(p_1=p_2=\id_{aa}\), this gives \(h=i_1\) and \(h=i_2\), a contradiction. Therefore \(aa\times aa\) does not exist.

For coproducts, we show that two copies of \(a\) have no coproduct in \(\Sub_{aa}\). A coproduct of \(a\) and \(a\) would be an object \(C\) with morphisms \(\iota_1,\iota_2\colon a\to C\) such that, for any object \(X\) and any pair \(f,g\colon a\to X\), there is a unique morphism
\([f,g]\colon C\to X\) satisfying
\[
[f,g]\circ\iota_1=f,
\qquad
[f,g]\circ\iota_2=g.
\]

Any candidate \(C\) must be either \(a\) or \(aa\). If \(C=a\), then
\(\iota_1=\iota_2=\id_a\). Taking \(X=aa\), \(f=i_1\), and \(g=i_2\), we would need \(i_1=i_2\), impossible.

If \(C=aa\), then each \(\iota_j\colon a\to aa\) is either \(i_1\) or \(i_2\). If \(\iota_1\neq\iota_2\), assume without loss of generality that \((\iota_1,\iota_2)=(i_1,i_2)\). Taking \(X=aa\) and \(f=g=i_1\), the only morphism \(aa\to aa\) is the identity, but it does not send both \(i_1\) and \(i_2\) to \(i_1\).

If \(\iota_1=\iota_2=i_j\), choose \(k\neq j\), take \(X=aa\), and set \(f=i_k\), \(g=i_j\). Then a morphism \([f,g]\colon aa\to aa\) would have to satisfy
\[
[f,g]\circ i_j=i_k
\quad\text{and}\quad
[f,g]\circ i_j=i_j,
\]
which is impossible. Thus no coproduct of \(a\) and \(a\) exists in \(\Sub_{aa}\).

\subsubsection{Permutations of the word}

In the Boolean case all letters are distinct, and any permutation of the positions yields an isomorphic subword category, also isomorphic to the boolean lattice \(2^{[n]}\). This is no longer true when
letters repeat.

\begin{example}
Let \(w=aba\) and \(w'=aab\). Their subword categories have different sets of objects:
\[
\begin{array}{c|c}
\Sub_{aba} & \Sub_{aab} \\ \hline
\varnothing & \varnothing \\
a & a \\
b & b \\
aa & aa \\
ab & ab \\
ba & \text{(none)} \\
aba & aab
\end{array}
\]
The word \(ba\) belongs to \(\Sub_{aba}\) but not to \(\Sub_{aab}\). Hence the two categories cannot be isomorphic; since they are acyclic, they are not even equivalent.
\end{example}

One might ask whether the grouped form
\[
w^*=a_1^{m_1}\cdots a_k^{m_k},
\]
where identical letters are grouped together, admits a faithful embedding into the original \(\Sub_w\), or conversely. The following example shows that such an embedding need not exist in one direction, whereas the other case is a consequence of the previous example.

\begin{example}
Let \(w=baba\), and let its grouped form, in first-occurrence order, be \(w^*=bbaa\). The object \(ba\) has exactly four inclusions into \(bbaa\): choose one of the two \(b\)'s and one of the two \(a\)'s independently to obtain the four cases 
\[
(1,3),\quad (1,4),\quad (2,3), \quad (2,4).
\]

In \(\Sub_{baba}\), the only object of length \(4\) is \(baba\), and the number of inclusions of \(ba\) into \(baba\) is \(3\), namely
\[
(1,2),\quad (1,4),\quad (3,4).
\]
Any faithful functor \(\Sub_{bbaa}\to\Sub_{baba}\) would have to send the maximal object \(bbaa\) to an object of length \(4\), hence to \(baba\). It would then induce an injective map
\[
\Hom(ba,bbaa)\longrightarrow \Hom(E(ba),baba),
\]
which is impossible. The domain has cardinality \(4\) and the codomain has cardinality at most \(3\). Indeed, by a non very tedious verification for every object \(u\in\Sub_{baba}\),
\[
|\Hom_{\Sub_{baba}}(u,baba)|\le 3.
\]
where the maximum is attained for \(u=ba\), with inclusions
\[
(1,2),\quad (1,4),\quad (3,4).
\] The one-letter objects \(a\) and \(b\) have two inclusions into \(baba\), while all other objects have at most one. Hence the codomain has cardinality at most \(3\), contradicting faithfulness. Thus \(\Sub_{bbaa}\) does not embed faithfully into \(\Sub_{baba}\).
\end{example}

\subsection{Decomposition as a product of thickened chains}

For a word consisting of a single letter repeated \(m\) times, \(w=x^m\), the subword category \(\Sub_{x^m}\) admits a simple description. Its objects are
\[
x^0,x^1,\dots,x^m,
\]
where \(x^0\) is the empty word, and a morphism \(x^i\to x^j\), for \(i\leq j\), is a strictly increasing injection \([i]\to[j]\). There are exactly
\[
\binom{j}{i}
\]
such morphisms.

\begin{remark}\label{re:FinInj}
This means that \(\Sub_{x^m}\) is isomorphic to the full subcategory of
\(\FinInj_{\mathrm{ord}}\) whose objects are the standard ordinals \([0],[1],\dots,[m]\). Here \(\FinInj_{\mathrm{ord}}\) denotes the category whose objects are finite totally ordered sets and whose morphisms are strictly increasing maps. The labelling by the single letter \(x\) can be forgotten.
\end{remark}

When the word has several distinct letters grouped contiguously, the subword category decomposes as a product of these thickened chains.

\begin{prop}\label{prop:product-decomposition}
Let
\[
w=a_1^{m_1}a_2^{m_2}\cdots a_k^{m_k},
\]
where the \(a_i\) are distinct. Then
\[
\Sub_w
\cong
\Sub_{a_1^{m_1}}\times \Sub_{a_2^{m_2}}\times\cdots\times
\Sub_{a_k^{m_k}}.
\]
\end{prop}

\begin{proof}
Any subword of \(w\) is uniquely determined by choosing, for each letter \(a_i\), a subword of the block \(a_i^{m_i}\). These choices are independent because the letters are distinct and the blocks do not interleave. A morphism between two such subwords is a family of strictly increasing injections, one for each block. Hence the assignment
\[
(u_1,\dots,u_k)\longmapsto u_1\cdots u_k
\]
defines an isomorphism between the product category and \(\Sub_w\).
\end{proof}

When all \(m_i=1\), each factor \(\Sub_{a_i}\) is the two-element chain \(0<1\), viewed as a category. Hence the product is the Boolean lattice \(2^{[k]}\). Thus the weak \(2\)-dimension of posets corresponds to the minimal length of a word without repeated letters needed for an injective order-preserving representation, while the general theory replaces the factor \(0<1\) by \textit{thickened chains} carrying parallel morphisms that are represented by the categories \(\Sub_{x^m}\).

\section{When does a category embed into a word?}

We now characterize the finite acyclic categories that admit weak embeddings into categories of subwords. Throughout the rest of the paper, \emph{word-embeddable} means weakly word-embeddable unless explicitly stated otherwise.

\begin{definition}
Let \(\CC\) be a finite acyclic category. The \emph{weak word dimension} of \(\CC\), denoted \(\operatorname{wwdim}(\CC)\), is the least non-negative integer \(n\) for which there exist an alphabet \(\Sigma\), a word \(w\) over \(\Sigma\) of length \(n\), and a weak embedding
\[
\CC\longrightarrow \Sub_w.
\]
If no such word exists, we set \(\operatorname{wwdim}(\CC)=\infty\).
\end{definition}

\begin{remark}
One can similarly define a \emph{strong word dimension} by requiring the representing functor to be fully faithful. For posets, the strong version recovers the usual \(2\)-dimension, while the weak version recovers \(\dim^{\mathrm{w}}_2\). The present paper characterizes the weak version.
\end{remark}

\subsection{The monic condition and the fork category}

\begin{definition}
A morphism \(f\colon c\to d\) in a category is a \emph{monomorphism}, or \emph{monic}, if for any two morphisms \(g,h\colon b\to c\),
\[
f\circ g=f\circ h \quad\Longrightarrow\quad g=h.
\]
A category \(\CC\) is \emph{monic} if every morphism is a monomorphism.
\end{definition}

In the category of subwords, every morphism is an injective order-preserving map, and hence is monic.

\begin{lemma}\label{lem:faithful-reflects-monic}
Let \(F\colon\CC\to\Dc\) be a faithful functor. If \(F(f)\) is monic, then
\(f\) is monic.
\end{lemma}

\begin{proof}
Suppose \(f\circ g=f\circ h\). Applying \(F\), we get
\[
F(f)\circ F(g)=F(f)\circ F(h).
\]
Since \(F(f)\) is monic, \(F(g)=F(h)\). Since \(F\) is faithful, \(g=h\).
\end{proof}

\begin{corollary}
If \(\CC\) weakly embeds into \(\Sub_w\), then \(\CC\) is monic.
\end{corollary}

\begin{proof}
All morphisms in \(\Sub_w\) are monic, and faithful functors reflect monomorphisms by Lemma~\ref{lem:faithful-reflects-monic}.
\end{proof}

The failure of monicity can be detected by a small category.

\begin{definition}
The \emph{fork category} \(\Fork\) is the category generated by the diagram
\[
\begin{tikzcd}
x \arrow[r, "f", bend left] \arrow[r, "g"', bend right]
&
y \arrow[r, "h"]
&
z
\end{tikzcd}
\]
with the relation
\[
h\circ f=h\circ g.
\]
\end{definition}

\begin{prop}[Characterization of monic acyclic categories]
Let \(\CC\) be an acyclic category. Then \(\CC\) is monic if and only if there is no weak embedding \(\Fork\to\CC\).
\end{prop}

\begin{proof}
If such a functor exists, then the image of \(h\) is not monic, so \(\CC\) is not monic.

Conversely, suppose \(\CC\) is not monic. Then there exist a morphism \(h\colon y\to z\) and two distinct morphisms \(f,g\colon x\to y\) such that
\[
h\circ f=h\circ g.
\]
Since \(\CC\) is acyclic, the objects \(x,y,z\) are distinct. Sending the generators of \(\Fork\) to \(f,g,h\) defines a faithful functor
\(\Fork\to\CC\) which is injective on objects.
\end{proof}

\subsection{Left locally totally ordered categories}
\begin{definition}
A small category \(\CC\) is \emph{left locally totally ordered} if, for each pair of objects \(c,d\), the hom-set \(\Hom(c,d)\) is equipped with a total order such that postcomposition strictly preserves the order. That is, for every morphism \(g\colon d\to e\), the map
\[
g\circ -\colon \Hom(c,d)\longrightarrow \Hom(c,e)
\]
is strictly order-preserving:
\[
f_1<f_2 \quad\Longrightarrow\quad g\circ f_1<g\circ f_2.
\]
\end{definition}

\begin{example}\label{ex:P}
Let \(\mathcal{P}\) be the category with three objects \(x,y,z\), arrows
\[
\begin{tikzcd}
x \arrow[r, "f_1", bend left] \arrow[r, "f_2"', bend right]
&
y \arrow[r, "g_1", bend left] \arrow[r, "g_2"', bend right]
&
z
\end{tikzcd}
\]
and relations
\[
g_1\circ f_1=g_2\circ f_2,
\qquad
g_2\circ f_1=g_1\circ f_2.
\]
This category cannot be left locally totally ordered.

Indeed, suppose first that \(f_1<f_2\) in \(\Hom(x,y)\). By postcomposing with \(g_1\) and \(g_2\), we get
\[
g_1\circ f_1<g_1\circ f_2,
\qquad
g_2\circ f_1<g_2\circ f_2.
\]
Using the defining relations, these inequalities become
\[
g_2\circ f_2<g_1\circ f_2,
\qquad
g_1\circ f_2<g_2\circ f_2,
\]
a contradiction. The case \(f_2<f_1\) is analogous.
\end{example}

\subsection{Main theorem}

We now prove the necessary and sufficient conditions for a finite acyclic category to be weakly embeddable into a category of words.

\begin{theorem}\label{thm:embed}
Let \(\CC\) be a finite acyclic category. There exists an alphabet \(\Sigma\) and a weak embedding
\[
F\colon \CC\longrightarrow \Word(\Sigma)
\]
if and only if \(\CC\) is monic and left locally totally ordered.
\end{theorem}

\begin{proof}
\((\Rightarrow)\)
Assume that there exists an alphabet \(\Sigma\) and a weak embedding
\[
F\colon \CC\longrightarrow \Word(\Sigma).
\]
Let
\[
U_\Sigma\colon \Word(\Sigma)\longrightarrow \FinInj_{\mathrm{ord}}
\]
be the forgetful functor that sends a word \(w\colon[n]\to\Sigma\) to the underlying finite ordered set \([n]\), and a subword inclusion to its underlying strictly increasing injection. The functor \(U_\Sigma\) is
faithful. Hence the composite
\[
G=U_\Sigma\circ F\colon \CC\longrightarrow \FinInj_{\mathrm{ord}}
\]
is faithful.

In \(\FinInj_{\mathrm{ord}}\), every morphism is an injective function, hence a monomorphism. Since faithful functors reflect monomorphisms, every morphism of \(\CC\) is monic.

We now construct the local orders. Equip each hom-set of \(\FinInj_{\mathrm{ord}}\) with the lexicographic order: for two distinct strictly increasing injections \(\alpha,\beta\colon[m]\to[n]\), let \(i\) be the first index such that \(\alpha(i)\neq\beta(i)\), and set
\[
\alpha<\beta
\quad\Longleftrightarrow\quad
\alpha(i)<\beta(i).
\]
This order is total. Moreover, it is preserved by postcomposition: if \(h\colon[n]\to[k]\) is strictly increasing and \(\alpha<\beta\), then
\[
h\circ\alpha<h\circ\beta,
\]
because the first index where the two maps differ is still \(i\), and
\(h(\alpha(i))<h(\beta(i))\).

Pull this order back along \(G\). For \(f,g\in\Hom_{\CC}(x,y)\), define
\[
f<g
\quad\Longleftrightarrow\quad
G(f)<G(g).
\]
Since \(G\) is faithful, this gives a total order on \(\Hom_{\CC}(x,y)\). The postcomposition compatibility follows from the corresponding property in \(\FinInj_{\mathrm{ord}}\). Hence \(\CC\) is left locally totally ordered.

\((\Leftarrow)\)
Assume that \(\CC\) is monic, acyclic, and left locally totally ordered. Choose a linear extension
\[
c_1<c_2<\cdots<c_N
\]
of the objects of \(\CC\), using Lemma~\ref{lem:linear-extension}. For each object \(c\), define a word \(w_c\) over the alphabet
\[
\Sigma=\{c_1,\dots,c_N\}
\]
by
\[
w_c
=
c_1^{|\Hom(c_1,c)|}
c_2^{|\Hom(c_2,c)|}
\cdots
c_N^{|\Hom(c_N,c)|},
\]
with the convention that \(x^0\) is the empty word.

We first show that \(c\mapsto w_c\) is injective on objects. Let \(c\neq d\). If
\[
\Hom(c,d)=\varnothing=\Hom(d,c),
\]
then the letter \(c\) appears in \(w_c\) but not in \(w_d\), while \(d\) appears in \(w_d\) but not in \(w_c\). Hence \(w_c\neq w_d\). Otherwise, without loss of generality, suppose \(\Hom(c,d)\neq\varnothing\). By acyclicity, \(\Hom(d,c)=\varnothing\). Then the letter \(d\) appears in \(w_d\) but not in \(w_c\), so again \(w_c\neq w_d\).

Now let \(f\colon c\to d\) be a morphism. For each object \(c_i\), write the elements of \(\Hom(c_i,c)\) in increasing order:
\[h_{i,1}<h_{i,2}<\cdots<h_{i,m_i},\qquad m_i=|\Hom(c_i,c)|.\]
Similarly, write
\[k_{i,1}<k_{i,2}<\cdots<k_{i,m_i'}, \qquad m_i'=|\Hom(c_i,d)|.\]
Postcomposition with \(f\) gives a map
\[f_*\colon \Hom(c_i,c)\longrightarrow \Hom(c_i,d),
\qquad h\longmapsto f\circ h. \]
Since \(f\) is monic, \(f_*\) is injective. Since \(\CC\) is left locally totally ordered, \(f_*\) is strictly order-preserving. Therefore there is a strictly increasing map
\[
\phi_i\colon \{1,\dots,m_i\}\longrightarrow \{1,\dots,m_i'\}
\]
such that
\[
f\circ h_{i,p}=k_{i,\phi_i(p)}
\]
for all \(p\).

Define \(F(f)\colon w_c\to w_d\) by sending the \(p\)-th occurrence of the letter \(c_i\) in \(w_c\) to the \(\phi_i(p)\)-th occurrence of \(c_i\) in \(w_d\). Since each \(\phi_i\) is strictly increasing, this is an injective order-preserving map respecting labels. Hence \(F(f)\) is a morphism in \(\Word(\Sigma)\).

For identities, the maps \(\phi_i\) are identities, so
\[
F(\id_c)=\id_{w_c}.
\]
If \(f\colon c\to d\) and \(g\colon d\to e\), then
\[
(g\circ f)_*=g_*\circ f_*,
\]
and hence
\[
F(g\circ f)=F(g)\circ F(f).
\]
Thus \(F\) is a functor.

Finally, \(F\) is faithful. Let \(f,g\colon c\to d\) be distinct morphisms. Consider the block corresponding to \(c_i=c\). Since \(\CC\) is acyclic,
\[\Hom(c,c)=\{\id_c\}.\]
The unique occurrence of the letter \(c\) in \(w_c\) is sent by \(F(f)\) to the position of \(f=f\circ\id_c\) in the \(c\)-block of \(w_d\), and by \(F(g)\) to the position of \(g=g\circ\id_c\). Since \(f\neq g\), these positions are distinct. Hence \(F(f)\neq F(g)\).

Thus \(F\) is a weak embedding.
\end{proof}

\begin{corollary}\label{cor:subword-embedding}
Let \(\CC\) be a finite acyclic category. Then
\[
\operatorname{wwdim}(\CC)<\infty
\]
if and only if \(\CC\) is monic and left locally totally ordered.
\end{corollary}

\begin{proof}
If \(\CC\) weakly embeds into \(\Sub_w\), then it weakly embeds into
\(\Word(\Sigma)\), so Theorem~\ref{thm:embed} applies.

Conversely, the construction in Theorem~\ref{thm:embed} assigns to each object \(c\) a word \(w_c\). Since \(\CC\) is finite, there are only finitely many such words. Let \(w\) be their concatenation, in any order. Then each \(w_c\) is a subword of \(w\), and therefore the functor constructed in Theorem~\ref{thm:embed} lands in \(\Sub_w\). Hence \(\CC\) has finite weak word dimension.
\end{proof}

\begin{remark}
The construction in the proof gives a weak representation. It is not, in general, a strong representation. Indeed, the representing words may admit additional subword inclusions that do not come from morphisms of \(\CC\). This is exactly the categorical analogue of the difference between an injective order-preserving map of posets and an order embedding.
\end{remark}

\begin{remark}
The particular embedding constructed above has the following useful property: if \(c\) and \(d\) are objects with
\[
\Hom(c,d)=\varnothing=\Hom(d,c),
\]
then
\[
\Hom(w_c,w_d)=\varnothing=\Hom(w_d,w_c).
\]
Thus the construction does not create comparabilities between objects that were incomparable in the underlying object-poset. However, it may still create extra parallel morphisms between comparable objects, and therefore it is not usually fully faithful.
\end{remark}

\begin{example}
Consider the free category \(P'\) on the diagram
\[
\begin{tikzcd}
x \arrow[r, "f_1", bend left] \arrow[r, "f_2"', bend right]
&
y \arrow[r, "g_1", bend left] \arrow[r, "g_2"', bend right]
&
z
\end{tikzcd}
\]
with the four composites \(h_{ij}=g_j\circ f_i\) all distinct.
The category \(P'\) is monic and acyclic. Choose
\[
f_1<f_2,
\qquad
g_1<g_2,
\]
and order \(\Hom(x,z)\) by
\[
h_{11}<h_{21}<h_{12}<h_{22}.
\]
This order is compatible with postcomposition by \(g_1\) and \(g_2\).

The construction of Theorem~\ref{thm:embed} gives
\[
w_x=x,\qquad w_y=xxy,\qquad w_z=xxxxyyz
\]
over the alphabet \(\{x,y,z\}\). The morphisms \(f_1\) and \(f_2\) are represented by the two inclusions of \(x\) into \(xx\). With the order chosen above, the morphisms \(g_1\) and \(g_2\) are represented by
\[
g_1=(1,3,5),
\qquad
g_2=(2,4,6).
\]
This weak representation is not full: the word \(xxy\) has more than two subword inclusions into \(xxxxyyz\).
\end{example}

\subsection{Free categories}

We now show that path categories of finite acyclic quivers are word-embeddable. We recall the basic terminology; see \cite{CatWorking} or \cite{Polygraph} for related material and connections with other combinatorial structures in category theory and higher rewriting systems.

A morphism \(f\) in a category is \emph{atomic} if it is not an identity and, whenever \(f=h\circ g\), either \(g\) or \(h\) is an identity. A category is \emph{free} if every morphism can be written uniquely as a composite of atomic morphisms; equivalently, it is the path category of a directed graph.

Every small free category is generated by a quiver \(Q=(Q_0,Q_1,s,t)\). The path category \(\operatorname{Path}(Q)\) has objects \(Q_0\), and a morphism from \(v\) to \(w\) is a finite directed path
\[
e_1e_2\cdots e_n
\]
with \(s(e_1)=v\), \(t(e_i)=s(e_{i+1})\), and \(t(e_n)=w\). The empty path at \(v\) gives the identity \(\id_v\). Composition is concatenation of paths.

\begin{prop}
Let \(Q\) be a finite acyclic quiver and let
\[
\CC=\operatorname{Path}(Q).
\]
Then \(\CC\) is monic and left locally totally ordered. Consequently, \(\CC\) is weakly word-embeddable.
\end{prop}

\begin{proof}
Every morphism in \(\CC\) is a monomorphism because path concatenation is cancellative on the appropriate side: if
\[
f\circ g=f\circ h,
\]
then \(g=h\) by uniqueness of path decomposition.

To build the left local total order, fix a topological ordering of the vertices
\[
v_1<v_2<\cdots<v_n,
\]
(this is possible because \(Q\) is acyclic). For each vertex \(v\), choose an arbitrary total order on the edges with source \(v\). For paths with common source and target, order them lexicographically from the left: if
\[
p=e_1p',
\qquad
q=e_2q',
\]
and \(e_1\neq e_2\), compare \(e_1\) and \(e_2\) using the chosen order on edges with that source; if \(e_1=e_2\), compare \(p'\) and \(q'\) recursively. This gives a total order on every hom-set.

We write paths in the order in which they are traversed. Thus, if \(p\colon x\to y\) and \(h\colon y\to z\), then \(h\circ p\) is obtained from \(p\) by appending the path \(h\) on the right. Therefore \(h\circ p\) and \(h\circ q\) have the same first difference as \(p\) and \(q\): postcomposition appends the same suffix to both paths. Hence the lexicographic order from the left is preserved under postcomposition. This proves that \(\CC\) is left locally totally ordered.
\end{proof}

\subsection{The poset case}

We briefly spell out what the construction of Theorem~\ref{thm:embed} gives when the category is a poset and how it relates with previous constructions. First, we can apply the theorem. Every poset is monic since there are no parallel arrows and it is also left locally totally ordered since every hom-set has at most one morphism.

Recall that every finite poset embeds into a Boolean lattice by means of principal downsets \cite{Trotter_Embedding}:
\[
P\longrightarrow 2^P,
\qquad
p\longmapsto \downarrow p=\{q\in P\mid q\leq p\}.
\]
This map is injective and order-preserving. In fact, it is also order reflecting: if
\[
\downarrow p\subseteq \downarrow q,
\]
then using that \(p\in\downarrow p\) we obtain that \(p\in\downarrow q\), and therefore \(p\leq q\). Thus the principal downset map embeds \(P\) as a subposet of \(2^P\). This give us a way to embed a poset into the subwords of a word by identifying \(2^P\) with the category of subwords of the word given by concatening all the objects of the poset.

We will see that the principal downsets construction is the same as the one that appears in the proof of Theorem~\ref{thm:embed}.

\begin{lemma}\label{lem:poset-downset-construction}
Let \(P\) be a finite poset, regarded as a category. Then the construction of Theorem~\ref{thm:embed} sends each object \(p\in P\) to the word
\[
w_p=p_1^{|\Hom(p_1,p)|}p_2^{|\Hom(p_2,p)|}\cdots
p_N^{|\Hom(p_N,p)|},
\]
which is precisely the subword of \(w=p_1p_2\cdots p_N\) corresponding to the principal downset \(\downarrow p\).
\end{lemma}

\begin{proof}
When \(P\) is regarded as a category, there is exactly one morphism
\(p_i\to p\) if \(p_i\leq p\), and no morphism otherwise. Hence
\[
|\Hom(p_i,p)|
=
\begin{cases}
1, & p_i\leq p,\\
0, & p_i\nleq p.
\end{cases}
\]
Therefore the word constructed in Theorem~\ref{thm:embed} is
\[
w_p=p_1^{|\Hom(p_1,p)|} p_2^{|\Hom(p_2,p)|} \cdots p_N^{|\Hom(p_N,p)|},
\]
which contains the letter \(p_i\) exactly when \(p_i\leq p\). Thus \(w_p\) is exactly the subword of \(w\) determined by the subset
\[
\{p_i\in P\mid p_i\leq p\}=\downarrow p.
\]
This is precisely the principal downset representation.
\end{proof}

\begin{remark}
In the poset case, the construction is stronger than merely weak. Since the word \(w=p_1\cdots p_N\) has no repeated letters, \(\Sub_w\) is a poset, and for \(p,q\in P\) we have
\[
w_p\leq w_q
\quad\Longleftrightarrow\quad
\downarrow p\subseteq \downarrow q
\quad\Longleftrightarrow\quad
p\leq q.
\]
Thus the principal downset construction gives a strong embedding of \(P\) into \(\Sub_w\cong 2^P\)..
\end{remark}

\subsection{Stability under categorical constructions}

The class of categories that embed weakly into a word, equivalently finite acyclic monic left locally totally ordered categories, is closed under several natural operations.

\subsubsection{Slice and coslice categories}

Let \(\CC\) satisfy the embedding criterion and let \(c\in\CC\). The slice category \(\CC/c\) has as objects pairs \((x,f)\), where \(f\colon x\to c\).
A morphism
\[
h\colon (x,f)\to(y,g)
\]
is a morphism \(h\colon x\to y\) in \(\CC\) such that
\[
g\circ h=f.
\]
Similarly we can define the \emph{coslice category} \(c\backslash \CC\), also denoted \(c\downarrow\CC\), such that it has as objects pairs \((x,f)\), where \(f\colon c\to x\) is a morphism of \(\CC\). A morphism
\[
h\colon (x,f)\to (y,g)
\]
is a morphism \(h\colon x\to y\) in \(\CC\) such that
\[
h\circ f=g.
\]
\begin{prop}
If \(\CC\) is finite, acyclic, monic, and left locally totally ordered, then so is \(\CC/c\) and \(c\backslash \CC\). Hence \(\CC/c\) and \(c\backslash \CC\) are weakly word-embeddable.
\end{prop}

\begin{proof}
Finiteness is clear. Acyclicity follows because any cycle in \(\CC/c\) projects
to a cycle in \(\CC\).

For monicity, suppose
\[
h\colon (x,f)\to(y,g)
\]
and two morphisms into \((x,f)\) become equal after postcomposition with \(h\). The equality holds in \(\CC\), and since \(h\) is monic in \(\CC\), the two morphisms are equal.

For the local order, observe that
\[
\Hom_{\CC/c}((x,f),(y,g))
=
\{h\colon x\to y\mid g\circ h=f\}
\]
is a subset of \(\Hom_{\CC}(x,y)\). Give it the restricted order. If
\(h_1<h_2\) and
\[
m\colon (y,g)\to(z,p)
\]
is a morphism in \(\CC/c\), then \(m\colon y\to z\) and \(p\circ m=g\). Since
\(\CC\) is left locally totally ordered,
\[ m\circ h_1<m\circ h_2. \]
Moreover,
\[
p\circ(m\circ h_i)=(p\circ m)\circ h_i=g\circ h_i=f,
\]
so \(m\circ h_i\) lies in the corresponding hom-set of \(\CC/c\). Thus \(\CC/c\) is left locally totally ordered. 

The proof for the case of \(c\backslash \CC\) is almost the same. 
\end{proof}

\subsubsection{Functor categories with acyclic domain}

Let \(I\) be a finite acyclic category and let \(\CC\) satisfy the embedding criterion. The functor category \([I,\CC]\) has as objects functors \(F\colon I\to\CC\) and as morphisms natural transformations.

\begin{prop}
If \(I\) is finite acyclic and \(\CC\) is finite, acyclic, monic, and left locally totally ordered, then \([I,\CC]\) has the same properties. Hence \([I,\CC]\) is weakly word-embeddable.
\end{prop}

\begin{proof}
Finiteness is clear.

Acyclicity: suppose that there are natural transformations
\(\eta\colon F\Rightarrow G\) and \(\mu\colon G\Rightarrow F\). For each object \(i\in I\), we have morphisms
\[
\eta_i\colon F(i)\to G(i),
\qquad
\mu_i\colon G(i)\to F(i).
\]
Since \(\CC\) is acyclic, these force \(F(i)=G(i)\) and
\(\eta_i=\mu_i=\id_{F(i)}\). Naturality of \(\eta\), whose components are identities, then gives \(F(\alpha)=G(\alpha)\) for every morphism \(\alpha\) of \(I\). Hence \(F=G\) and \(\eta=\id_F\).

Monicity: let \(\eta\colon F\Rightarrow G\) be a natural transformation, and let \(\alpha,\beta\colon H\Rightarrow F\) satisfy
\[
\eta\circ\alpha=\eta\circ\beta.
\]
Then for every object \(i\in I\),
\[
\eta_i\circ\alpha_i=\eta_i\circ\beta_i.
\]
Since every morphism of \(\CC\) is monic, \(\eta_i\) is monic, and hence \(\alpha_i=\beta_i\). Therefore \(\alpha=\beta\), so \(\eta\) is monic.

For the local order, fix a linear extension
\[
i_1<i_2<\cdots<i_n
\]
of the objects of \(I\). For two functors \(F,G\), order natural transformations \(\eta,\mu\colon F\Rightarrow G\) lexicographically: set \(\eta<\mu\) if, at the first index \(i_k\) for which
\(\eta_{i_k}\neq\mu_{i_k}\), one has
\[
\eta_{i_k}<\mu_{i_k}
\]
in \(\Hom_{\CC}(F(i_k),G(i_k))\).

Now let \(\eta<\mu\colon F\Rightarrow G\), and let \(\beta\colon G\Rightarrow K\). At the first index \(i_k\) where \(\eta\) and
\(\mu\) differ, left local order preservation in \(\CC\) gives
\[
\beta_{i_k}\circ\eta_{i_k} < \beta_{i_k}\circ\mu_{i_k}.
\]
For earlier indices the components remain equal. Hence
\[
\beta\circ\eta<\beta\circ\mu.
\]
Thus postcomposition preserves the lexicographic order.
\end{proof}

\subsubsection{Finite limits}

\begin{prop}
The class of finite acyclic monic left locally totally ordered categories is closed under finite limits. Consequently, any finite limit of weakly word-embeddable categories is weakly word-embeddable.
\end{prop}

\begin{proof}
Every finite limit can be built from finite products, equalizers and a terminal object \cite{Basic_Category,CatWorking}. For categories the terminal category is the category with only one object and one morphism, the identity of that object. Clearly the terminal category is weakly word-embeddable.

For finite products, acyclicity and monicity are checked componentwise. The left local total order is the lexicographic order on products of hom-sets, and postcomposition acts componentwise, preserving that order.

For equalizers, let
\[
F,G\colon\mathcal C\to \mathcal D
\]
be parallel functors. Their equalizer is the subcategory of \(\mathcal C\) whose objects \(a\) satisfy \(F(a)=G(a)\), and whose morphisms \(u\colon a\to a'\) satisfy \(F(u)=G(u)\). This subcategory inherits acyclicity, monicity, and the left local total order by restriction.
\end{proof}

\section{Query games and separators}\label{sec:queries}

We now introduce a game that captures the combinatorial data used in the construction of weak word representations.

\subsection{The Boolean query game for posets}

Consider a finite poset \(P\). Two players engage in the following game. Player \(A\) secretly chooses an element \(p\in P\). Player \(B\) must determine \(p\) by asking questions of the form
\[
\text{``Is }q\leq p\text{?''}
\]
for certain elements \(q\in P\) selected in advance.

A \emph{winning set} is a subset \(W=\{p_1,\dots,p_n\}\subseteq P\) such that, for every \(p\in P\), the answers to the questions
\[
\text{``Is }p_i\leq p\text{?''}\qquad (i=1,\dots,n)
\]
uniquely determine \(p\). Equivalently, the map
\[
p\longmapsto \{\,p_i\in W\mid p_i\leq p\,\}
\] is injective.

The \emph{query dimension} of \(P\), denoted \(\mathrm{qdim}(P)\), is the minimum cardinality of a winning set.

\begin{prop}
For any finite poset \(P\),
\[
\dim^{\mathrm{w}}_2(P)\leq \mathrm{qdim}(P).
\]
\end{prop}

\begin{proof}
Let \(W=\{p_1,\dots,p_n\}\) be a winning set. For each \(p\in P\), define
\[
W_p=\{\,p_i\in W\mid p_i\leq p\,\}.
\]
Since \(W\) is winning, the assignment \(p\mapsto W_p\) is injective. If \(p\leq q\), then every element of \(W\) lying below \(p\) also lies below \(q\), so \(W_p\subseteq W_q\). Thus
\[
p\longmapsto W_p
\]
is an injective order-preserving map from \(P\) to \(2^W\cong 2^{[n]}\). Therefore \(\dim^{\mathrm{w}}_2(P)\leq n\). Taking \(n=\mathrm{qdim}(P)\) gives the result.
\end{proof}

\subsection{The categorical query game}

Let \(\CC\) be a finite category. The game proceeds as follows:
\begin{enumerate}
\item Player \(A\) secretly chooses a morphism \(f\colon x\to y\).
\item Player \(B\) selects a finite set \(B=\{b_1,\dots,b_n\}\) of test
objects.
\item First, \(B\) asks for each \(t_i\) the cardinalities
\[
|\Hom(b_i,x)|,
\qquad
|\Hom(b_i,y)|.
\]
If this suffices to determine \(x\) and \(y\), the game proceeds.
\item Then, for each \(t_i\in B\) and each morphism \(h\colon t_i\to x\),
Player \(B\) asks for the composite \(f\circ h\).
\end{enumerate}

A set \(B\) is \emph{winning} if Player \(B\) can always determine the secret morphism \(f\) uniquely.

\begin{definition}
Let \(\CC\) be a category. A set \(B\subseteq\Ob(\CC)\) is
\begin{itemize}
\item a \emph{cardinal separator} if the map
\[
X\longmapsto (|\Hom(b,x)|)_{b\in B}
\]
is injective on objects;
\item a \emph{separator} if for every pair of distinct parallel morphisms \(f,g\colon \allowbreak x\to y\), there exist \(b\in B\) and \(h\colon b\to x\) such that
\[f\circ h\neq g\circ h.\]
\end{itemize}
\end{definition}

\begin{theorem}\label{thm:winning}
A set \(B\) is winning if and only if it is both a cardinal separator and a separator.
\end{theorem}

\begin{proof}
If \(B\) is winning, the first phase must distinguish objects, so the
cardinality vectors are distinct. The second phase must distinguish parallel morphisms; otherwise two distinct morphisms with the same endpoints would produce identical answers. Hence \(B\) is both a cardinal separator and a separator.

Conversely, suppose \(B\) has both properties. The first phase determines \(X\) and \(Y\). Once the endpoints are known, the answers in the second phase test candidate morphisms \(X\to Y\). Since \(B\) separates parallel morphisms, at most one morphism can produce those answers. Hence \(B\) is winning.
\end{proof}

When \(\CC\) is a poset, the second condition is vacuous, and the first condition reduces to the Boolean query game described above.

\begin{theorem}\label{thm:game-embedding}
Let \(\CC\) be a finite acyclic category which is weakly word-embeddable, and let \(B\subseteq\Ob(\CC)\) be a winning set. Then there exists an alphabet
\(\Sigma=B\) and a word \(w\) such that
\[
\CC\longrightarrow \Sub_w
\]
is a weak embedding. The length of \(w\) can be taken to be
\[
\sum_{b\in B}\max_{c\in\Ob(\CC)}|\Hom(b,c)|.
\]
\end{theorem}

\begin{proof}
Since \(\CC\) is weakly word-embeddable, Corollary~\ref{cor:subword-embedding} implies that \(\CC\) is monic and left locally totally ordered. Fix such a left local total order on its hom-sets. Write
\[
B=\{b_1,\dots,b_r\}
\]
in some fixed order by a linear extension of
\(\CC\). For each object \(c\in\Ob(\CC)\), define
\[
m_j(c)=|\Hom(b_j,c)|.
\]
We associate to \(c\) the word
\[
w_c =b_1^{m_1(c)}b_2^{m_2(c)} \cdots b_r^{m_r(c)}
\]
over the alphabet \(B\).

We first show that the assignment \(c\mapsto w_c\) is injective on objects. Since \(B\) is winning, it is in particular a cardinal separator. Hence, if \(c\neq d\), then
\[
\bigl(|\Hom(b,c)|\bigr)_{b\in B}\neq \bigl(|\Hom(b,d)|\bigr)_{b\in B}.
\]
Equivalently, \(m_j(c)\neq m_j(d)\) for at least one \(j\). Therefore the words \(w_c\) and \(w_d\) have different block lengths, and so
\[
w_c\neq w_d.
\]
Thus the assignment is injective on objects.

Now let \(f\colon c\to d\) be a morphism of \(\CC\). We define a subword inclusion
\[
F(f)\colon w_c\to w_d.
\]
Fix \(j\in\{1,\dots,r\}\). Since each hom-set is totally ordered, write
\[
\Hom(b_j,c)=\{h_{j,1}<h_{j,2}<\cdots<h_{j,m_j(c)}\},
\]
and
\[
\Hom(b_j,d)=\{k_{j,1}<k_{j,2}<\cdots<k_{j,m_j(d)}\}.
\]
Postcomposition with \(f\) gives a map
\[
f_*\colon \Hom(b_j,c)\longrightarrow \Hom(b_j,d),
\qquad
h\longmapsto f\circ h.
\]
Since \(\CC\) is monic, \(f_*\) is injective. Since \(\CC\) is left locally totally ordered, \(f_*\) is strictly order-preserving. Therefore there is a unique strictly increasing map
\[
\phi_j^f\colon
\{1,\dots,m_j(c)\}
\longrightarrow
\{1,\dots,m_j(d)\}
\]
such that
\[
f\circ h_{j,p}=k_{j,\phi_j^f(p)}
\]
for every \(p\).

We now define \(F(f)\) blockwise. The \(p\)-th occurrence of the letter \(b_j\) in \(w_c\) is sent to the \(\phi_j^f(p)\)-th occurrence of the letter \(b_j\) in \(w_d\). Since each \(\phi_j^f\) is strictly increasing, and since the block order \(b_1,\dots,b_r\) is the same in \(w_c\) and \(w_d\), this defines an injective order-preserving map of positions
\[
F(f)\colon [|w_c|]\longrightarrow [|w_d|].
\]
Moreover, it preserves labels by construction. Hence \(F(f)\) is a subword inclusion \(w_c\to w_d\).

This assignment is functorial. If \(f=\id_c\), then postcomposition with \(\id_c\) is the identity on every \(\Hom(b_j,c)\), so each
\(\phi_j^{\id_c}\) is the identity map. Hence
\[
F(\id_c)=\id_{w_c}.
\]
If \(f\colon c\to d\) and \(g\colon d\to e\), then for every \(j\) we have
\[
(g\circ f)_*=g_*\circ f_*.
\]
Consequently,
\[
\phi_j^{g\circ f}=\phi_j^g\circ \phi_j^f.
\]
It follows block by block that
\[
F(g\circ f)=F(g)\circ F(f).
\]
Thus \(F\colon\CC\to\Word(B)\) is a functor.

We next prove that \(F\) is faithful. Let
\[
f,g\colon c\to d
\]
be distinct parallel morphisms. Since \(B\) is winning, it is a separator. Therefore there exist \(b_j\in B\) and \(h\colon b_j\to c\) such that
\[
f\circ h\neq g\circ h.
\]
Write \(h=h_{j,p}\) in the ordered list of \(\Hom(b_j,c)\). Under \(F(f)\), the \(p\)-th occurrence of \(b_j\) in \(w_c\) is sent to the occurrence of \(b_j\) in \(w_d\) corresponding to \(f\circ h\). Under \(F(g)\), it is sent to the occurrence corresponding to \(g\circ h\). These two morphisms are distinct, hence they occupy different positions in the ordered list of
\(\Hom(b_j,d)\). Therefore
\[
F(f)\neq F(g).
\]
Thus \(F\) is faithful.

It remains to ensure that the image lands in a single subword category \(\Sub_w\) with the stated length. For each \(j\), set
\[
M_j=\max_{c\in\Ob(\CC)}|\Hom(b_j,c)|.
\]
Define
\[
w=b_1^{M_1}b_2^{M_2}\cdots b_r^{M_r}.
\]
Then, for every object \(c\), the word
\[
w_c=b_1^{m_1(c)}\cdots b_r^{m_r(c)}
\]
is a subword of \(w\), because \(m_j(c)\leq M_j\) for all \(j\). Hence the functor \(F\) lands in the full subcategory \(\Sub_w\subseteq\Word(B)\).

Finally,
\[
|w|=\sum_{j=1}^r M_j=\sum_{b\in B}\max_{c\in\Ob(\CC)}|\Hom(b,c)|.
\]
Therefore \(F\colon\CC\to\Sub_w\) is a weak embedding with the claimed word length.
\end{proof}

Thus every winning set \(B\) gives an explicit upper bound
\[
\operatorname{wwdim}(\CC)
\leq
\sum_{b\in B}\max_{c\in\Ob(\CC)}|\Hom(b,c)|.
\]
It is therefore natural to introduce the weighted query bound
\[
\operatorname{wq}(\CC)= \min_{B\text{ winning}} \sum_{b\in B}\max_{c\in\Ob(\CC)}|\Hom(b,c)|.
\]
Then
\[
\operatorname{wwdim}(\CC)\leq \operatorname{wq}(\CC).
\]
The unweighted quantity \(\mathrm{qdim}(\CC)\), the minimum size of a winning set, does not by itself bound the weak word dimension in general, because a single test object may contribute several occurrences of a letter.

\begin{example}
Let \(\mathrm{Sur}\) be the category generated by
\[
\begin{tikzcd}
a \arrow[r, "f"]
&
b \arrow[r, "g_1", bend left] \arrow[r, "g_2"', bend right]
&
c
\end{tikzcd}
\]
with the relation
\[
g_1\circ f=g_2\circ f.
\]
This category is monic and left locally totally ordered, and hence is weakly word-embeddable.

A direct application of Theorem~\ref{thm:embed}, using the linear extension
\(a<b<c\), gives a weak representation by
\[
w_a=a,
\qquad
w_b=ab,
\qquad
w_c=abbc.
\]
The morphisms may be represented as
\[
f=(1),
\qquad
g_1=(1,2),
\qquad
g_2=(1,3),
\]
and indeed \(g_1\circ f=g_2\circ f\).

Now consider the query game. The singleton \(B=\{b\}\) is winning:
\begin{itemize}
\item \(|\Hom(b,a)|=0\), \(|\Hom(b,b)|=1\), and \(|\Hom(b,c)|=2\), so
the map \(X\mapsto |\Hom(b,X)|\) is injective on objects;
\item the only nontrivial parallel morphisms are \(g_1,g_2\colon b\to c\),
and
\[
g_1\circ\id_b=g_1\neq g_2=g_2\circ\id_b,
\]
so \(B\) separates them.
\end{itemize}
Theorem~\ref{thm:game-embedding} therefore gives a weak representation over
the alphabet \(\{b\}\) using a word of length
\[
\max_{x\in\Ob(\mathrm{Sur})}|\Hom(b,x)|=2.
\]
Thus \(\mathrm{Sur}\) weakly embeds into \(\Sub_{bb}\). A word of length \(1\)
cannot represent the two distinct morphisms \(g_1,g_2\), so
\[
\operatorname{wwdim}(\mathrm{Sur})=2.
\]
\end{example}
\begin{example}
Let \(\CC\) be the category generated by the following diagram
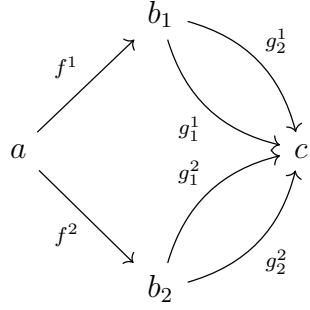
\[
\begin{tikzcd}[column sep=3em, row sep=3em]
                                       & b_1 \arrow[rd, "g_2^1", bend left] \arrow[rd, "g_1^1"', bend right] &   \\
a \arrow[ru, "f^1"] \arrow[rd, "f^2"'] &                                                                     & c \\
                                       & b_2 \arrow[ru, "g^2_1", bend left] \arrow[ru, "g^2_2"', bend right] &  
\end{tikzcd}
\]
with relations
\[
g_1^1\circ f^1=g_1^2\circ f^2,
\qquad
g_2^1\circ f^1=g_2^2\circ f^2,
\]
and no further relations. We claim that
\[
B=\{a,b_1,b_2\}
\]
is a winning set.

First we check that \(B\) is a cardinal separator. The hom-set cardinality
vectors, ordered with respect to \(B=(a,b_1,b_2)\), are as follows:
\[
\begin{array}{c|c}
x & \bigl(|\Hom(a,x)|,\ |\Hom(b_1,x)|,\ |\Hom(b_2,x)|\bigr) \\ \hline
a   & (1,0,0) \\
b_1 & (1,1,0) \\
b_2 & (1,0,1) \\
c   & (2,2,2).
\end{array}
\]
Indeed, the two morphisms \(a\to c\) are
\[
g_1^1\circ f^1=g_1^2\circ f^2
\qquad\text{and}\qquad
g_2^1\circ f^1=g_2^2\circ f^2,
\]
while the two morphisms \(b_1\to c\) are \(g_1^1,g_2^1\), and the two
morphisms \(b_2\to c\) are \(g_1^2,g_2^2\). Since the four vectors displayed
above are distinct, \(B\) is a cardinal separator.

We now check that \(B\) is a separator. The only nontrivial parallel pairs of
morphisms are
\[
g_1^1,g_2^1\colon b_1\to c,
\qquad
g_1^2,g_2^2\colon b_2\to c,
\]
and the two composites from \(a\) to \(c\):
\[
h_1:=g_1^1\circ f^1=g_1^2\circ f^2,
\qquad
h_2:=g_2^1\circ f^1=g_2^2\circ f^2.
\]
Every pair is separated by composing with the identity morphism of the domain. Thus \(B\) is a separator. 

Now we will use Theorem~\ref{thm:game-embedding} to obtain a word embedding.
For \(B=(a,b_1,b_2)\), we have
\[
\max_x |\Hom(a,x)| =\max_x |\Hom(b_1,x)|= \max_x |\Hom(b_2,x)|=2.
\]
Therefore the theorem gives a word of length \(6\), namely
\[
w=aab_1b_1b_2b_2.
\]
The associated words are
\[
w_a=a,\qquad
w_{b_1}=ab_1,\qquad
w_{b_2}=ab_2,\qquad
w_c=aab_1b_1b_2b_2.
\]
For instance, if we order
\[
h_1<h_2,\qquad
g_1^1<g_2^1,\qquad
g_1^2<g_2^2,
\]
then the morphisms may be represented by the following subword inclusions:
\[
F(f^1)=(1)\colon a\to ab_1,
\qquad
F(f^2)=(1)\colon a\to ab_2,
\]
\[
F(g_1^1)=(1,3),\qquad
F(g_2^1)=(2,4),
\]
and
\[
F(g_1^2)=(1,5),\qquad F(g_2^2)=(2,6).
\]
With these choices,
\[
F(g_1^1)\circ F(f^1)=F(g_1^2)\circ F(f^2),
\qquad
F(g_2^1)\circ F(f^1)=F(g_2^2)\circ F(f^2),
\]
as required.

In this case we have only a weak embedding since, for example, there are more than \(2\) morphisms between \(F(b_1)=ab_1\) and \(F(c)=aab_1b_1b_2b_2\).
\end{example}
\bibliographystyle{plain}
\bibliography{biblio}

\end{document}